\documentclass[11pt]{article}
\usepackage{amssymb}
\usepackage{amsmath}

\newtheorem{precor}{{\bf Corollary}}

\newenvironment{cor}{\begin{precor}{\hspace{-0.5
               em}{\bf.\ }}}{\end{precor}}
\newtheorem{precon}{{\bf Conjecture}}

\newenvironment{con}{\begin{precon}{\hspace{-0.5
               em}{\bf.\ }}}{\end{precon}}
\newtheorem{predefin}{{\bf Definition}}

\newtheorem{preexm}{{\bf Example}}

\newtheorem{preappl}{{\bf Application}}

\newtheorem{prelem}{{\bf Lemma}}

\newtheorem{preproof}{{\bf Proof.\ }}

\newenvironment{proof}[1]{\begin{preproof}{\rm
               #1}\hfill{$\blacksquare$}}{\end{preproof}}
\newtheorem{preclm}{{\bf Claim}}

\newenvironment{clm}{\begin{preclm}{\hspace{-0.5
               em}{\bf.\ }}}{\end{preclm}}
\newtheorem{prethm}{{\bf Theorem}}

\newenvironment{thm}{\begin{prethm}{\hspace{-0.5
               em}{\bf.\ }}}{\end{prethm}}
\newtheorem{prealphthm}{{\bf Theorem}}

\newenvironment{alphthm}{\begin{prealphthm}{\hspace{-0.5
               em}{\bf.\ }}}{\end{prealphthm}}
\newtheorem{prealphlem}{{\bf Lemma}}

\newtheorem{prepro}{{\bf Proposition}}

\newenvironment{pro}{\begin{prepro}{\hspace{-0.5
               em}{\bf.\ }}}{\end{prepro}}
\newtheorem{prequ}{{\bf Question}}

\newtheorem{prealphqu}{{\bf Question}}

\newtheorem{preprb}{{\bf Problem}}

\def\conct[#1,#2]{\mbox {${#1} \leftrightarrow {#2}$}}
\def\dconct[#1,#2]{\mbox {${#1} \rightarrow {#2}$}}
\def\deg[#1,#2]{\mbox {$d_{_{#1}}(#2)$}}
\def\mindeg[#1]{\mbox {$\delta_{_{#1}}$}}
\def\maxdeg[#1]{\mbox {$\Delta_{_{#1}}$}}
\def\outdeg[#1,#2]{\mbox {$d_{_{#1}}^{^+}(#2)$}}
\def\minoutdeg[#1]{\mbox {$\delta_{_{#1}}^{^+}$}}
\def\maxoutdeg[#1]{\mbox {$\Delta_{_{#1}}^{^+}$}}
\def\indeg[#1,#2]{\mbox {$d_{_{#1}}^{^-}(#2)$}}
\def\minindeg[#1]{\mbox {$\delta_{_{#1}}^{^-}$}}
\def\maxindeg[#1]{\mbox {$\Delta_{_{#1}}^{^-}$}}

\def\dre[#1,#2,#3]{\mbox {${\cal E}_{_{#3}}(#1,#2)$}}
\def\var[#1,#2]{\mbox {${\rm Var}_{_{#1}}(#2)$}}
\def\ls[#1]{\mbox {$\xi^{^{#1}}$}}
\def\hom[#1,#2]{\mbox {${\rm Hom}({#1},{#2})$}}
\def\onvhom[#1,#2]{\mbox {${\rm Hom^{v}}(#1,#2)$}}
\def\onehom[#1,#2]{\mbox {${\rm Hom^{e}}(#1,#2)$}}
\def\core[#1]{\mbox {$#1^{^{\bullet}}$}}
\def\cay[#1,#2]{\mbox {${\rm Cay}({#1},{#2})$}}
\def\cays[#1,#2]{\mbox {${\rm Cay_{s}}({#1},{#2})$}}
\def\dirc[#1]{\mbox {$\stackrel{\rightarrow}{C}_{_{#1}}$}}
\def\cycl[#1]{\mbox {${\bf Z}_{_{#1}}$}}

\date{}

\begin{document}
\begin{center}
{\Large \bf $r$-Dynamic Chromatic Number  of  Graphs }\\
\vspace*{0.5cm}
{\bf  Ali Taherkhani}\\
{\it Department of Mathematics}\\
{\it Institute for Advanced Studies in Basic Sciences }\\
{\it P.O. Box {\rm 45195-1159}, Zanjan {\rm 45195}, Iran}\\
{\tt ali.taherkhani@iasbs.ac.ir}\\
\end{center}
\begin{abstract}
\noindent 
An $r$-dynamic $k$-coloring of a graph $G$ is a proper vertex $k$-coloring such that the neighbors of any vertex $v$
receive at least $\min\{r,{\rm deg}(v)\}$ different colors. The $r$-dynamic chromatic number of  $G$, $\chi_r(G)$, 
is defined as the smallest $k$ such that $G$ admits an $r$-dynamic $k$-coloring. In this paper we introduce  an upper
bound for $\chi_r(G)$ in terms of $r$, chromatic number, maximum degree and minimum degree.   
In 2001, Montgomery  \cite{MR2702379} conjectured that, for a $d$-regular graph $G$, $\chi_2(G)-\chi(G)\leq 2$.
In this regard, for a $d$-regular graph $G$, we present two  upper bounds for $\chi_2(G)-\chi(G)$,
  one of them,  $\lceil 5.437\log d+2.721\rceil$, is an improvement  of the bound $14.06\log d +1$, proved by Alishahi (2011) \cite{MR2746973}.
  Also, we give an upper bound for $\chi_2(G)$ in terms of chromatic number, maximum degree and minimum degree.   
\\

\end{abstract}
\section{Introduction}

Let $r$ be a positive integer. An  $r$-dynamic $k$-coloring of a graph $G$ is a proper vertex $k$-coloring such that 
 every vertex $v$ receives at least $\min\{r,\rm{ deg}(v)\}$ colors  in its neighbors. 
The minimum $k$ for which a graph $G$ has a $r$-dynamic $k$-coloring is called  $r$-dynamic chromatic number of  $G$, 
 and denoted $\chi_r(G)$. It is obvious that $\chi(G)\leq\chi_r(G)\leq \chi_{r+1}(G)$. 
 The $r$-dynamic chromatic number first introduced by  Montgomery \cite{MR2702379}.
In case  $r=2$,   it is called dynamic chromatic number. The dynamic chromatic number, $\chi_2(G)$, have been  investigated  in several papers,  
see, e.g., \cite{MR2935408,MR2746973,MR2954752,MR2954756,MR3057028,MR1991048}.
 In 2001  Montgomery conjectured that for a regular graph $G$, $\chi_2(G)-\chi(G)\leq 2$.
\begin{con}{\rm \cite{MR2702379}}\label{mg}
For every regular graph $G$, we have $\chi_2(G)-\chi(G)\leq 2.$
\end{con}
 Some upper bounds for the dynamic chromatic number of graphs and regular graphs have been studied in recent years.
\begin{alphthm}{\rm \cite{MR2702379}}\label{mont}
Let $G$ be a graph with maximum degree $\Delta(G)$. Then $\chi_2(G)\leq \Delta(G)+3.$ 
\end{alphthm}
In this regard, for a graph $G$ with $\Delta(G)\geq3$, it was proved that $\chi_2(G)\leq \Delta(G)+1$ \cite{MR1991048}. Also, for a regular graph $G$, it was shown by  Alishahi:
\begin{alphthm}{\rm {\cite{MR2954752}}}\label{mysm}
If $G$ is a $d$-regular graph,  then $\chi_2(G)\leq \chi(G)+14.06\log d+1.$ 
\end{alphthm} 

 Alishahi proved that for every graph $G$ with $\chi(G)\geq 4$, $\chi_2(G)\leq \chi(G)+\gamma(G)$, 
 where $\gamma(G)$ is the domination number of a graph $G$  \cite{MR2954752}.  
In the proof of current upper bound for $\chi_2(G)$, it was shown that if $\chi(G)\geq 4$,
 then for every $k\geq \chi(G)$, there is a proper $k$-coloring $c$ such that the 
set of vertices whose neighbor set receives only one color is an independent set.  
For a $k$-coloring $c$ of $G$, we say that a vertex $v$ of $G$ is bad if  ${\rm deg}(v)\geq 2$ and only one color appears in the neighbors of $v$.  
Let $B_c$ be the set of all bad vertices in the $k$-coloring $c$ of $G$. 

Another upper bound for the dynamic chromatic number of a $d$-regular graph $G$ in terms of $\chi(G)$ and the independence number of  $G$, $\alpha(G)$, was
introduced in \cite{MR2954756}. In fact, it was proved that $\chi_2(G)\leq \chi(G)+2\log_2\alpha(G)+3$. 

The set of neighbors of a vertex $v$ in a
graph $G=(V(G),E(G))$ is denoted by $N_G(v)$ (or $N(v)$) and  the degree of $v$ in $G$, $|N_G(v)|$, is denoted by ${\rm deg}(v)$.
For $B\subseteq V(G)$, denote by $N_G(B)$(or $N(B)$) the set of neighbors of $B$ in $G$. 
The second power of a  graph $G$ is the graph $G^2$ whose vertex set is
$V(G)$,  but in which two distinct vertices are adjacent  if and only if their distance in $G$ is
at most $2$.
In this paper $\log$ stands for the natural logarithm and $e$ denotes its base.

Here we recall two powerful probabilistic tools which we will use in proofs. For more details of the following theorems, we refer the reader to
Appendix A  and Chapter 5 of \cite{MR1885388}.
\begin{alphthm}
{\rm (Chernoff Inequality). }
Let $X_1,\ldots, X_n$ be independent random variables. They need not have the same distribution. Assume that $0\leq X_i\leq1$ always, for each $i$.
Let $X=X_1 + \cdots + X_n$ and 
$\mu=E(X)=E(X_1)+\cdots+E(X_n)$.
Then for any $\epsilon\geq 0$,
$$\Pr(X \geq(1+\epsilon)\mu)\leq{ ({e^\epsilon\over{(1+\epsilon)^{1+\epsilon}}})^\mu},$$
$$\Pr(X \leq(1-\epsilon)\mu)\leq \exp (-\frac{\epsilon^2}{2}\mu).$$
\end{alphthm}
\begin{alphthm}{\rm (Lov\'asz  Local Lemma).}
Suppose that $A_1,\ldots, A_n$
are events in a probability space with $\Pr(A_i)\leq p$ for all $i$. If each event is mutually
independent of all the other events except for at most $d$ of them, and if $ep(d+1)\leq 1$,
then $$\Pr(\cap_{i=1}^n A_i)>0.$$
\end{alphthm}

\section{Results}

In this section we present some upper bounds on the dynamic chromatic number and  the $r$-dynamic chromatic number  in terms of the chromatic number,
the maximum degree, and the minimum degree.  

First we prove a proposition, which gives an upper bound for the $r$-dynamic chromatic number by applying the following theorem. 
\begin{alphthm}{\rm\cite{MR1446766}}\label{mc}
Let $H$ be a hypergraph in which every hyperedge contains
at least $k$ points and meets at most $d$ other hyperedges. Let $r\geq 2$ be an integer. If
$e((d+1)(r-1)+1)(1-1/r)^k\leq 1$, then $H$ has an $r$-coloring in which each color
appears on each hyperedge.
\end{alphthm}
\begin{pro}\label{pro}
Let G be a graph with  maximum degree $\Delta$ and minimum  degree $\delta$. If  $e((\delta\Delta-\delta+1)(r-1)+1)(1-1/r)^{\delta}\leq1$, then $\chi_r(G)\leq r\chi(G)$.
\end{pro} 
\begin{proof}{
For every vertex $v$, fix a subset $N'(v)\subseteq N(v)$ with $\delta$ elements.
Consider a hypergraph $H$ whose vertex set is $V(G)$ and its hyperedge set is $E(H)=\{N'(v):v\in V(G)\}$.
Every hyperedge  meets at most $\delta(\Delta-1)$ other hyperedges. Therefore, as a consequence   of Theorem~\ref{mc}, 
 there is an $r$-coloring of $H$, say $g$, such that each color
appears on each hyperedge. Consider  a $\chi(G)$-coloring of $G$, say $f$. We can now obtain a $r$-dynamic coloring  
of $G$ by assigning to each
vertex $v$ the ordered pair $(f(v),g(v))$.}
\end{proof}

We should mention that Proposition~\ref{pro}, in case $r=2$, was proved in \cite{MR2746973} and as a corollary of this fact it was shown that  the dynamic chromatic number
of a $d$-regular $G$ with $d\geq 7$  is at most $2\chi(G)$. Also,  it was shown that for any $d$-regular graph $G$, $\chi_2(G)\leq 2\chi(G).$

 In \cite{MR2746973}, Alishahi  showed that each $d$-regular graph $G$ has dynamic chromatic number at most $\chi(G)+14.06\log d$+1. 
 Also, in \cite{MR2954752}, he proved that
 for any $d$-regular graph $G$ with no induced
$C_4$, $\chi_2(G)\leq \chi(G)+ 2\lceil4\log d +1\rceil$. 
We strengthen  these two results in the following theorem. We show that for a $d$-regular graph $G$, $\chi(G)+\lceil e\log(d^2+1)+e\rceil$, which is at most
$\chi(G)+\lceil 5.437\log d+2.721\rceil$, when $d\geq3$. 
\begin{thm}\label{Dynam1}
The dynamic chromatic number of  a $d$-regular graph $G$, $\chi_2(G)$, is at most
$$\chi(G)+\lceil e\log(d^2+1)+e\rceil.$$  
 \end{thm}
\begin{proof}{ Let   $G$ be a $d$-regular graph.
If $d\leq 14$, then $d+3\leq e(1+\log (d^2+1)).$
By Theorem \ref{mont},  $\chi_2(G)\leq d+3$ and as a result $\chi_2(G)\leq\chi(G)+\lceil e(1+\log(d^2+1))\rceil.$

Assume that  $\chi(G)\leq 3$ and $d\geq 14$. Since for a $d$-regular graph $G$, $\chi_2(G)\leq 2\chi(G)$, we have 
$\chi_2(G)\leq6$.
if $d\geq 14$, then $6\leq e(1+\log (d^2+1))$ and so
it is obtained that $\chi_2(G)\leq\chi(G)+e(1+\log (d^2+1))$.

Now suppose that  $d\geq 14$ and $\chi(G)\geq 4$. In view of  discussion  after Theorem~\ref{mysm}, there is a coloring, say $c$, with $\chi(G)$ colors 
such that the set of bad vertices $B_c$, for simplicity $B$, is independent.  
Choose every vertex of $N(B)$(neighborhoods of $B$) with probability $p={\log e(d^2+1)/ d}$ and put a set, say $D$.
We define three type bad events as follows.
\begin{itemize} 
\item For every $v\in B$, $E_v$ is event that $N(v)\cap D=\varnothing$ or $N(v)\subseteq D$.
\item For every $v\in N(B)$, $E_v$ is event that $ |N(v)\cap D|\geq edp$.
\item For every $v\in V(G)\setminus B\cup N(B)$, $E_v$ is event that $N(v)\subseteq D$.
\end{itemize} 
It is easy to see that for every  $v\in B$, $\Pr (E_v)=(1-p)^d+p^d$.  For $d\geq 14$, we have $$(1-p)^d+p^d\leq \exp(-pd)=\frac{1}{e(d^2+1)}.$$
Let $w\in N(B)$ and  $X_w$ be a random variable which counts the number of neighbors of $w$ in $D$. Hence,
 $$\Pr(E_w)=\Pr(X_w\geq edp).$$
Suppose that there is a vertex $b\in B$ with  $N(b)\subseteq B$.  Then, for every $w\in B$, $\Pr(X_w\geq l)\leq \Pr(X_b\geq l)$, where $l$ is a positive integer.
 Now by the Chernoff inequality, the probability 
$\Pr(X_w\geq (1+\epsilon)dp)\leq \Pr(X_b\geq (1+\epsilon)dp)\leq { ({e^\epsilon/{(1+\epsilon)^{1+\epsilon}}})^{dp}}$. Take $p={\log e(d^2+1)/ d}$ and $\epsilon=e-1$. Therefore,
$$\Pr (E_w)\leq {1\over e(d^2+1)}.$$
For every $u\in V(G)\setminus B\cup N(B)$, the probability of  event $E_u$ is at most $p^d$. 
Since $d\geq 14$, we have  $p^d<{1/ e(d^2+1)}$ and consequently $\Pr(E_u)< {1/ e(d^2+1)}$.

Every  event $E_v$ is mutually  independent of all the events  corresponding to vertices 
 which have distance greater than 2 from $v$. As there are  $d^2$ such events and $e(1/e(d^2+1))(d^2+1)\leq 1$, by the Lov\'asz  local lemma,
 it follows that $$\Pr({\displaystyle\bigcap_{v\in V(G)}}  E^c_v)>0.$$
 Therefore, there is a subset $D\subseteq N(B)$ which satisfies the following four conditions. First, every bad vertex $v$ has a neighbor in $D$ and a neighbor out of $D$.
  Second, the maximum degree of $G[D]$ is at most  $e(1+\log(d^2+1))$.
  Third, since $B\cap N(B)=\varnothing$, every $v\in N(B)$ has at least one neighbor out of $D$. 
  Fourth, every vertex $v\not\in B\cup N(B)$ has at least one neighbor out of $D$, too.
 
 Now, one can color $G[D]$ with $\lceil e(1+\log(d^2+1))\rceil$ new colors and extend this new coloring  and former coloring $c$ on vertices 
 $V(G)\setminus D$ to a proper dynamic coloring for $G$ with $\chi(G)+\lceil e(1+\log(d^2+1))\rceil$ colors, as desired.
}\end{proof}
Here we give another upper bound on $\chi_2(G)-\chi(G)$ for a $d$-regular graph $G$,
which is somewhat complementary to  the upper bound in \cite{MR2954752}. It was shown in \cite{MR2954752}  that the difference between the dynamic chromatic
number and the chromatic number of a $d$-regular graph $G$ with $\chi(G)\geq 4$ is at most $\alpha(G^2)$, where $\alpha(G^2)$ is the independence number
of graph $G^2$. 

Let $G$ be a graph and let $\{V_1, V_2,\ldots, V_n\}$ be a partition of $V(G)$  into $n$ pairwise disjoint sets. A transversal of $G$  is a subset  $T$ of vertices 
containing exactly one vertex from each $V_i$. We  apply the next theorem in the proof of Theorem~\ref{G2}.
\begin{alphthm}\label{forest}{\rm\cite{MR2246152}}
Let $H$ be a graph of maximum degree $d$ and $\{V_1, V_2,\ldots, V_n\}$ be a partition of its vertex set  into $n$ pairwise disjoint sets with
$|V_i|\geq  d$. Then
there is a transversal $T$ such that $H[T]$ is a forest.
\end{alphthm}
\begin{thm}\label{G2}
Let $G$ be a  $d$-regular graph and  $c$ be a $k$-coloring of  $G$. If $B_c$ is the set of all bad vertices for coloring $c$, then
$\chi_2(G)\leq k+2\chi(G^2[B_c]\setminus E(G))$. Therefore, $\chi_2(G)\leq \min\{k+2\chi(G^2[B_c]\setminus E(G)):\text{ c is a } k\text{-coloring of } G \}.$
\end{thm}
 \begin{proof}{
One can easily show that statement is true for $d=1,2$. Let $d\geq 3$.
Assume that $c$ be a $k$-coloring of  $G$ with the set of bad vertices $B_c$. 
By the definition of $B_c$, for each $v\in B_c$, only one color is assigned 
 to all neighbors of $v$, $N_G(v)$. 
 
 \noindent Consider the induced subgraph of $G^2$ on $B_c$, $G^2[B_c]$.
   Delet from $G^2[B_c]$ all  edges of $G$ that lie in $B_c$. 
Suppose that $l=\chi(G^2[B_c]\setminus E(G))$. Consider an $l$-coloring of $G^2[B_c]\setminus E(G)$ which partitions $B_c$ into $B_1,\ldots, B_l$. 
For every $B_i$, if 
$v,w\in B_i$, then $N_G(v)\cap N_G(w)=\varnothing$. It follows that $N_G(B_i)$ is partitioned to $\{N_G(v)|\, v\in B_i\}$. 

 \noindent Our aim is to find a subtest $T$ whose  intersection with the neighbor set  of each bad vertex is nonempty and 
  for which $G[T]$ has a $2l$-coloring. 

\noindent We say a bad vertex $v\in B_i$ is tractable if $v\in B_i\cap N_G(B_i)$. Also, we say a vertex $u\in V(G)\setminus B_c$ is potentially bad, if 
its neighbor set is an independent set and there is an $i$, $1\leq i\leq l$, $N_G(u)\subseteq N_G(B_i)$.
\begin{clm}\label{c1}
Every  vertex $u\in N_G(B_i)$ is adjacent to  at most one tractable vertex $v\in B_i$. Moreover, every tractable vertex $w\in B_i$ is adjacent to exactly one    
tractable vertex $v\in B_i$.
\end{clm}
 Assume that $v,v'\in B_i$ are two tractable vertices, which 
are adjacent to $u$. But this is a contradiction  because $u\in N_G(v)\cap N_G(v')$ and 
hence $v$ and $v'$ are adjacent in $G^2[B_c]\setminus E(G).$ 

 \begin{clm}\label{c2}
Let  $u$ be a potentially bad vertex. One of the following conditions is satisfied.
\begin{description}
\item[ a)] The vertex $u$ has at least two neighbors, say $x, x'$, such that both the number of neighbors of $x$ and the number of neighbors of $x'$
in $N_G(B_i)$ are  at most $d-1$.   
\item[ b)] The vertex $u$ has a neighbor in $N_G(B_i)$, say $y$, such that the number of neighbors of $y$ in $N_G(B_i)$ is equal to $d$.
This vertex, $y$, is not a tractable vertex and has a tractable vertex $v$ in its neighbor.  
\end{description}
\end{clm}
Let $u$ be a potentially bad vertex. If $u\not\in N_G(B_i)$, then the first condition is obviously hold. Suppose that  
$u\in N_G(B_i)$.  If condition (a)  is not hold, then at most one vertex in $N_G(u)$ has less than $d$ neighbors in $N_G(B_i)$.
Since $d\geq3$, $u$ is adjacent  to at least two vertices which have  $d$ neighbors in $N_G(B_i)$. 
Claim \ref{c1} yields that at least one of them  is not tractable, say $y$.
Therefore,  $y\in N_G(u)$ and $y$ has $d$ neighbors in $N_G(B_i)$.
There exists a vertex $v$ which belongs to $B_i$ and is adjacent to $y$
 because $y\in N_G(u)\subseteq N_G(B_i)$. 
 Also, $v\in N_G(B_i)$ because all of neighbors of $y$ lie in $N_G(B_i)$. Thus, $v$ is a tractable vertex.

 \noindent We are now ready to show that there is a transversal forest $T_i$ on vertices of $G[N_G(B_i)]$ which has a $2$-coloring $f_i$
 such that every bad vertex $v\in B_i$ has one neighbor in $T_i$ and 
  for a  potentially bad vertex $u$, if $N_G(u)\subseteq T_i$, then $|f_i(N_G(u))|\geq 2$. 
  
\noindent For every potentially bad vertex $u$, if it has at least two neighbors  with at most $d-1$
neighbors in $N_G(B_i)$, fix two of them and call $x_u, x'_u$.
For every $B_i$, we obtain a graph $H_i$  by adding all edges $x_ux'_u$ to induced subgraph $G[N_G(B_i)]$. 
If condition (a) is not hold for $u$, then $u$ has a neighbor $y$ which is not a tractable vertex and has a tractable vertex $v$ in its neighbor.
 By Claim~\ref{c1}, $v$ has a unique tractable vertex such $w$. Hence, $y,w\in N_G(v)$.
 
\noindent Assume that $S_i$ is the set of all tractable vertices in $H_i$. In view of Claim~\ref{c1}, we conclude that  induced subgraph of $H_i$ on $S_i$  is a matching.
 Consider the graph $H_i\setminus N_G(S_i)$. Since  $N_G(v)\cap N_G(w)=\varnothing$ for any two vertex $v,w\in B_i$,
 we obtain $\{N_G(v): v\in B_i\setminus S_i\}$ is a partition of $V(H_i)\setminus  N_G(S_i)$.
Since $G$ is a  $d$-regular graph, for each $v\in B_i$, we have $|N_G(v)|= d$. Also, one can show that the maximum degree
of  $H_i$ is at most $d$. In view of Theorem~\ref{forest},  there is a transversal $S'_i$ in $H_i\setminus  N_G(S_i)$ such that $S'_i$ is a forest and 
for every $v\in B_i\setminus S_i$, we have $|N_G(v)\cap S'_i|=1$.  Therefore, if take $T_i= S_i\cup S'_i$, 
then $T_i\cap N_G(v)\neq\varnothing$ for each vertex $v\in B_i$.
As $S'_i\subseteq H_i\setminus  N_G(S_i)$, we have  $N_G(S_i)\cap S'_i=\varnothing$.
Also, because induced subgraphs on $S_i$ and $S'_i$ in $H_i$ are forest and  $S_i, S'_i$ are disjoint, 
 one can conclude that $T_i$ is forest. Furthermore, if $u$ is  potentially bad vertex and $N_G(u)\subseteq T_i$,
  then for every $2$-coloring $f_i$ of $H_i[T_i]$, $|f_i(N_G(u))|\geq 2$.  

\noindent Set $T=\cup_{i=1}^lT_i $. Define a dynamic coloring $f$ with $k+2l$ colors for $G$ as follows.
If $v\not\in T$, define $f(v)=c(v)$; otherwise, define $f(v)=f_i(v)$, where $v\in T_i$. 
}
 \end{proof}
Note that if one can color a $d$-regular graph $G$ with $\chi(G)$ colors 
such that any two bad vertices of $G$  in this coloring have no common neighbors (i.e., have a distance either 1 or at least 3 with each other),
 then Theorem~\ref{G2} impleis that Montgomery's conjectre is true.

In \cite{MR2702379} Montgomery showed that $\chi_2(G)-\chi(G)$
for a graph $G$ can be arbitrarily large. It seems that the small difference between  $\Delta(G)$ and $\delta(G)$ provides a small  difference between
the dynamic chromatic number and the chromatic number of $G$.
\begin{thm}
Let $G$ be a graph with maximum degree $\Delta $ and minimum  degree $\delta$. We have  
$$\chi_2(G)-\chi(G)\leq \lceil e{\Delta\over\delta}\log(2e(\Delta^2+1))\rceil.$$
  \end{thm}
\begin{proof}{
 The validity of the theorem is readily verified when
$\Delta=1,2$. Thus, we may assume $\Delta\geq 3$.
First suppose that $\log 2e(\Delta^2+1)\leq \delta/2$ and $\chi(G)\geq 4$.
 Similar to the proof of Theorem~\ref{Dynam1}, assume that  there is a coloring, say $c$, with $\chi(G)$ colors 
such that the set of bad vertices is independent.  
Choose each vertex of $N(B)$(neighborhoods of $B$) with probability $p={\log 2e(\Delta^2+1)/ \delta}$ and put a set, say $D$.
Same as proof of Theorem~\ref{Dynam1}, we define three types of bad events as follows.
\begin{itemize} 
\item For every $v\in B$, $E_v$ is event that $N(v)\cap D=\varnothing$ or $N(v)\subseteq D$.
\item For every $v\in N(B)$, $E_v$ is event that $ |N(v)\cap D|\geq e{\rm deg(v)} p$.
\item For every $v\in V(G)\setminus(B\cup N(B))$, $E_v$ is event that $N(v)\subseteq D$.
\end{itemize} 
It is easy to see that for every  $v\in B$, $\Pr (E_v)\leq(1-p)^{\delta}+p^{\delta}$.  
Since $p\leq{1\over 2}$,
we have $(1-p)^{\delta}+p^{\delta}\leq 2(1-p)^{\delta}$ and so $$\Pr(E_v)\leq 2\exp(-p\delta)=\frac{1}{e(\Delta^2+1)}.$$

Let $w\in V(G)\setminus B$. By the same discussion of  proof of Theorem~\ref{Dynam1}, we can obtain   $\Pr(E_w)\leq {1/e(\Delta^2+1)}$.
Since every bad event $E_v$ is mutually independent  of all the events, but those which have distance at most 2 from $v$. Hence,  
$E_v$ is not mutually independent  of at most $\Delta^2$ bad events.
Therefore, by the Lov\'asz  local lemma,
 it follows that $$\Pr({\displaystyle\bigcap_{v\in V(G)}}E^c_v)>0.$$
  
 Now, one can color $G[D]$ with $\lceil e(\Delta/\delta)\log(2e(\Delta^2+1))\rceil$ new colors and extend this new coloring  and former coloring $c$ on vertices 
 $V(G)\setminus D$ to a proper dynamic coloring for $G$ with $\chi(G)+ \lceil e(\Delta/\delta)\log(2e(\Delta^2+1))\rceil$ colors, as desired.

Assume that  $\log 2e(\Delta^2+1)\leq \delta/2$ and $\chi(G)\leq 3$.  In this case, the assertion follows from Proposition~\ref{pro}.

If $\delta/2\leq {\log(2e(\Delta^2+1))}$, we have $e\Delta/2\leq e(\Delta/\delta)\log(2e(\Delta^2+1)).$
Also, since $\Delta\geq 3$, we obtain $\Delta+1\leq e\Delta/2$. Hence, $\chi_2(G)-\chi(G)\leq e(\Delta/\delta)\log(2e(\Delta^2+1)).$
}
\end{proof}
Here we establish a bound on the $r$-dynamic chromatic number of a graph $G$ in terms of $\chi(G)$, $\Delta(G)$, $\delta(G)$, and $r$, when
$r$ is at most ${\delta/\log(2er(\Delta^2+1))}.$

 \begin{thm}
Let $G$ be a graph with maximum degree $\Delta $ and minimum  degree $\delta$ and let $r$ be a positive integer with $2\leq r\leq {\delta/\log(2er(\Delta^2+1))}$.
Then the $r$-dynamic chromatic number of $G$, $\chi_r(G)$, is at most
  $$ \chi(G)+(r-1)\lceil e{\Delta\over\delta}\log(2er(\Delta^2+1))\rceil.$$ 
\end{thm}
\begin{proof}{
 Take $p={\log(2er(\Delta^2+1))/\delta}$. Let $Y$ be a random variable which takes  values $1,\ldots,r-1$ with probability $p$, and the value $r$ with probability 
$1-(r-1)p$. For every vertex $v$, choose randomly one element of $\{1,2,\ldots,r\}$ with probability distribution of $Y$. If $Y(v)=i$, put $v$ in $D_i$.    
Note that $V(G)$ is partitioned to $D_1, D_2, \ldots, D_r$.\\ 
For every vertex $v$, define  two bad events $A_v$ and $B_v$. Let  $A_v$ be the event that there exists a $D_i$ 
such that $N(v)\subseteq V(G)\setminus D_i$ and let $B_v$ be the event that there exists an $i\leq r-1$, 
$|N(v)\cap D_i|\geq e{\rm deg}(v)p$.
It is straight that  
$$\Pr(A_v)\leq ((r-1)p)^{{\rm deg}(v)}+(r-1)(1-p)^{{\rm deg}(v)}.$$
Since $p\leq 1/r$, one can see that 

$$
\begin{array}{ccc}
  ((r-1)p)^{{\rm deg}(v)}+(r-1)(1-p)^{{\rm deg}(v)}& \leq  &  r(1-p)^{{\rm deg}(v)} \\
         & &  \\
~&\leq   &   r(1-p)^{\delta}\\
       & &  \\
  ~&\leq   &   {1\over 2e(\Delta^2+1)}.
  \end{array}
$$
For $i\leq r-1$, suppose that $B_{v,i}$ is event that  $|N(v)\cap D_i|\geq e{\rm deg}(v)p$. Hence, $\Pr(B_v)\leq \Sigma_{i=1}^{r-1}\Pr(B_{v,i})$.
For each $i$, $1\leq i\leq r-1$, let $X_{v,i}$ be a random variable which  counts the number of neighbors of $v$ in $D_i$. Therefore, $\Pr(B_{v,i})=\Pr(X_{v,i}\geq e{\rm deg}(v)p)$. 
It is easy to see that   ${\rm E}(X_{v,i})= {\rm deg}(v)p$.
In view of the chernoff inequality,  we have 
$$\Pr(X_{v,i}\geq (1+\epsilon){\rm deg}(v)p)\leq { ({e^\epsilon/{(1+\epsilon)^{1+\epsilon}}})^{{\rm deg}(v)p}}$$
 and then 
$$Pr(B_v)\leq (r-1) { ({e^\epsilon/{(1+\epsilon)^{1+\epsilon}}})^{{\rm deg}(v)p}}.$$
If choose $\epsilon=e-1$, then   
$$
\begin{array}{ccc}
 \Pr(B_v)& \leq  & (r-1) \exp(-{\rm deg}(v)p) \\
         & &  \\
~&\leq   &   (r-1)\exp(-\delta p)\\
       & &  \\
  ~&\leq   &   {r-1\over 2re(\Delta^2+1)}\\
         & &  \\
  ~&\leq   &   {1\over 2e(\Delta^2+1)}.

  \end{array}
$$
It is clear that each  event $A_v$($B_v$) is mutually independent of all the events $A_w$ or $B_w$, but those which have distance at most 2 from $v$. 
 As there are at most $2\Delta^2+1$ such events,   the Lov\'asz  local lemma yields that  
 $$\Pr({\displaystyle\bigcap_{v\in V(G)}}(A^c_v\cap B^c_v))>0.$$
 Therefore, there exists  a partition $D_1, D_2, \ldots, D_r$ such that for each vertex $v\in V(G)$ and each $i$, $1\leq i\leq r$, $N(v)\cap D_i\neq\varnothing$.
Moreover, for  $1\leq i\leq r-1$, the maximum degree of $G[D_i]$ is at most 
$$\lfloor e{\Delta\over\delta}\log(2er(\Delta^2+1))\rfloor.$$
 Hence, by greedy  coloring algorithm,  
One can color each induced subgraphs $G[D_i]$ with $\lceil e{\Delta\over\delta}\log(2er(\Delta^2+1))\rceil$  colors. 
In addition, since $\chi(G[D_r])\leq \chi(G)$, we can color the vertices of $G[D_r]$ with at most $\chi(G)$ colors.
Hence,  
$$\chi_r(G)\leq \chi(G)+(r-1)\lceil e{\Delta\over\delta}\log(2er(\Delta^2+1))\rceil.$$ 
  } 
 \end{proof}
\begin{cor}
 Let $G$ be a $d$-regular graph and $r$ be a positive integer with $2\leq r\leq {d/\log (2er(d^2+1))}$.
Then the $r$-dynamic chromatic number of $G$, $\chi_r(G)$, is at most
$$\chi(G)+(r-1)\lceil e\log 2er(d^2+1)\rceil.$$
\end{cor}
\noindent {\bf Acknowledgement:} I would like to express my deep gratitude to Professor Carsten Thomassen for his support during my visit in
 Technical University of Denmark. My visit was supported by the ERC Advanced Grant GRACOL. I would also like to thank Professor Hossein Hajiabolhassan for his useful comments.

\begin{thebibliography}{10}

\bibitem{MR2935408}
A.~Ahadi, S.~Akbari, A.~Dehghan, and M.~Ghanbari,
\newblock On the difference between chromatic number and dynamic chromatic
  number of graphs,
\newblock {\em Discrete Math.}, 312(17):2579--2583, 2012.

\bibitem{MR2746973}
M.~Alishahi,
\newblock On the dynamic coloring of graphs,
\newblock {\em Discrete Appl. Math.}, 159(2-3):152--156, 2011.

\bibitem{MR2954752}
M.~Alishahi,
\newblock Dynamic chromatic number of regular graphs,
\newblock {\em Discrete Appl. Math.}, 160(15):2098--2103, 2012.

\bibitem{MR1885388}
N.~Alon and J.H.~Spencer,
\newblock {\em The probabilistic method},
\newblock Wiley-Interscience Series in Discrete Mathematics and Optimization.
  Wiley-Interscience [John Wiley \& Sons], New York, second edition, 2000,
\newblock With an appendix on the life and work of Paul Erd{\H{o}}s.

\bibitem{MR2954756}
A.~Dehghan and A.~Ahadi,
\newblock Upper bounds for the 2-hued chromatic number of graphs in terms of
  the independence number,
\newblock {\em Discrete Appl. Math.}, 160(15):2142--2146, 2012.

\bibitem{MR3057028}
S.J.~Kim, S.J.~Lee, and W.J.~Park,
\newblock Dynamic coloring and list dynamic coloring of planar graphs,
\newblock {\em Discrete Appl. Math.}, 161(13-14):2207--2212, 2013.

\bibitem{MR1991048}
H.J.~Lai, B.~Montgomery, and H.~Poon,
\newblock Upper bounds of dynamic chromatic number,
\newblock {\em Ars Combin.}, 68:193--201, 2003.

\bibitem{MR1446766}
C.~McDiarmid,
\newblock Hypergraph colouring and the {L}ov\'asz local lemma,
\newblock {\em Discrete Math.}, 167/168:481--486, 1997,
\newblock 15th British Combinatorial Conference (Stirling, 1995).

\bibitem{MR2702379}
B.~Montgomery,
\newblock {\em Dynamic coloring of graphs},
\newblock ProQuest LLC, Ann Arbor, MI, 2001,
\newblock Thesis (Ph.D.)--West Virginia University.

\bibitem{MR2246152}
T.~Szab{\'o} and G.~Tardos,
\newblock Extremal problems for transversals in graphs with bounded degree,
\newblock {\em Combinatorica}, 26(3):333--351, 2006.

\end{thebibliography}

\end{document}